\documentclass[10pt]{amsart}
\usepackage{amssymb,MnSymbol}
\usepackage{amsthm,amsmath}

\title[Influence and interaction indexes]{Influence and interaction indexes for pseudo-Boolean functions: a unified least squares approach}

\author{Jean-Luc Marichal}
\address{Mathematics Research Unit, FSTC, University of Luxembourg, 6, rue Coudenhove-Kalergi, L-1359 Luxembourg, Luxembourg}
\email{Jean-luc.marichal[at]uni.lu}

\author{Pierre Mathonet}
\address{University of Li\`ege, Department of Mathematics, Grande Traverse, 12 - B37, B-4000 Li\`ege, Belgium}
\email{p.mathonet[at]ulg.ac.be }

\date{March 29, 2014}

\begin{document}

\theoremstyle{plain}
\newtheorem{theorem}{Theorem}[section]
\newtheorem{lemma}[theorem]{Lemma}
\newtheorem{proposition}[theorem]{Proposition}
\newtheorem{corollary}[theorem]{Corollary}
\newtheorem{fact}[theorem]{Fact}
\newtheorem*{main}{Main Theorem}

\theoremstyle{definition}
\newtheorem{definition}[theorem]{Definition}
\newtheorem{example}[theorem]{Example}

\theoremstyle{remark}
\newtheorem*{conjecture}{onjecture}
\newtheorem{remark}{Remark}
\newtheorem{claim}{Claim}

\newcommand{\N}{\mathbb{N}}
\newcommand{\R}{\mathbb{R}}
\newcommand{\I}{\mathbb{I}}
\newcommand{\Vspace}{\vspace{2ex}}
\newcommand{\bfx}{\mathbf{x}}
\newcommand{\bfy}{\mathbf{y}}
\newcommand{\bfh}{\mathbf{h}}
\newcommand{\bfe}{\mathbf{e}}
\newcommand{\p}{\mathbf{p}}
\newcommand{\bfp}{\mathbf{p}}
\newcommand{\Imp}{\mathcal{I}}
\newcommand{\Q}{Q}
\newcommand{\Sy}{{\mathcal{S}}}
\newcommand{\A}{{\mathcal{A}}}

\begin{abstract}
The Banzhaf power and interaction indexes for a pseudo-Boolean function (or a cooperative game) appear
naturally as leading coefficients in the standard least squares approximation of the function by a pseudo-Boolean function of a specified degree. We first observe that this property still holds if we consider approximations by pseudo-Boolean functions depending only on specified variables. We then show that the Banzhaf influence index can also be obtained from the latter approximation problem. Considering certain weighted versions of this
approximation problem, we introduce a class of weighted Banzhaf influence indexes, analyze their most important properties, and point out similarities between the weighted Banzhaf influence index and the corresponding weighted Banzhaf interaction index. We also discuss the issue of reconstructing a pseudo-Boolean function from prescribed influences and point out very different behaviors in the weighted and non-weighted cases.
\end{abstract}



\maketitle

\section{Introduction}

Let $f\colon\{0,1\}^n\to\R$ be an $n$-variable pseudo-Boolean function and let $S$ be a subset of its variables. Define the \emph{influence of $S$ over $f$} as the expected value, denoted $I_f(S)$, of the highest variation of $f$ when assigning values independently and uniformly at random to the variables not in $S$ (see \cite{Mar00} for a normalized version of this definition). That is,
$$
I_f(S) ~=~ \frac{1}{2^{n-|S|}}\,\sum_{T\subseteq N\setminus S}\Big(\max_{R\subseteq S}f(T\cup R)-\min_{R\subseteq S}f(T\cup R)\Big)\, ,
$$
where $N=\{1,\ldots,n\}$.\footnote{Throughout we identify Boolean vectors $\bfx\in\{0,1\}^n$ and subsets $T\subseteq N$ by setting $x_i=1$ if and only if $i\in T$. We thus use the same symbol to denote both a pseudo-Boolean function $f\colon\{0,1\}^n\to\R$ and the corresponding set function $f\colon 2^{N}\to\R$ interchangeably.} This notion was first introduced for Boolean functions $f\colon\{0,1\}^n\to\{0,1\}$ by Ben-Or and Linial \cite{BenLin90} (see also \cite{KahKalLin88}). There the influence $I_f(S)$ was (equivalently) defined as the probability that, assigning values independently and uniformly at random to the variables not in $S$, the value of $f$ remains undetermined. Since its introduction, this concept has found many applications in discrete mathematics, cooperative game theory, theoretical computer science, and social choice theory (see, e.g., the survey article \cite{KalSaf06}).

When the function $f$ is nondecreasing in each variable, the formula above reduces to
\begin{equation}\label{eq:sd76dfs}
I_f(S) ~=~ \frac{1}{2^{n-|S|}}\,\sum_{T\subseteq N\setminus S}\big(f(T\cup S)-f(T)\big)\, .
\end{equation}
The latter expression has an interesting interpretation even if $f$ is not nondecreasing. In cooperative game theory for instance, where $f(T)$ represents the worth of coalition $T$ in the game $f$, this expression is precisely the average value of the marginal contributions $f(T\cup S)-f(T)$ of coalition $S$ to outer coalitions $T\subseteq N\setminus S$. Thus, it measures an overall influence (which can be positive or negative) of coalition $S$ in the game $f$. In particular, when $S=\{i\}$ is a singleton it reduces to the Banzhaf power index
$$
I_f(\{i\}) ~=~ \frac{1}{2^{n-1}}\,\sum_{T\subseteq N\setminus\{i\}}\big(f(T\cup\{i\})-f(T)\big)\, .
$$
Thus, the expression in (\ref{eq:sd76dfs}) can be seen as a variant of the original concept of influence that simply extends the Banzhaf power index to coalitions. We call it the \emph{Banzhaf influence index} and denote it by $\Phi_{\mathrm{B}}(f,S)$. Actually, this index was introduced, axiomatized, and even generalized to weighted versions in \cite{MarKojFuj07}.

The \emph{Banzhaf interaction index} \cite{Rou96}, another index which extends the Banzhaf power index to coalitions, is defined for a pseudo-Boolean function $f\colon\{0,1\}^n\to\R$ and a subset $S\subseteq N$ by
\begin{equation}\label{eq:d7df5fd}
I_{\mathrm{B}}(f,S) ~=~ \frac{1}{2^{n-|S|}}\,\sum_{T\subseteq N\setminus S}(\Delta_Sf)(T)\, ,
\end{equation}
where $\Delta_Sf$ denotes the $S$-difference (or discrete $S$-derivative) of $f$.\footnote{The differences of $f$ are defined as
$\Delta_{\varnothing}f=f$, $\Delta_{\{i\}}f(\mathbf{x})=f(\mathbf{x}\mid x_i=1)-f(\mathbf{x}\mid x_i=0)$, and $\Delta_Sf=\Delta_{\{i\}}\Delta_{S\setminus\{i\}}f$ for $i\in S$.} When $|S|\geqslant 2$, this index measures an overall degree of interaction among the variables of $f$ that are in $S$. When $f$ is a game, it measures an overall degree of interaction among the players of coalition $S$ in the game $f$ (see, e.g., \cite{FujKojMar06,GraMarRou00,GraRou99}).

It is known that the Banzhaf power and interaction indexes can be obtained from the solution of a standard least squares approximation problem for pseudo-Boolean functions (see \cite{GraMarRou00,HamHol92}). Weighted versions of this approximation problem recently enabled us to define a class of weighted Banzhaf interaction indexes having several nice properties (see \cite{MarMat11}). However, we observe that there is no such least squares construction for the Banzhaf influence index in the literature.

In this paper we fill this gap in the following way. In Section 2 we first show that the Banzhaf interaction index can be obtained from a different, more natural (but still elementary) least squares approximation problem. Specifically, $I_{\mathrm{B}}(f,S)$ appears as the leading coefficient in the multilinear representation of the best approximation $f_S$ of $f$ by a pseudo-Boolean function that depends only on the variables in $S$. We then prove that the Banzhaf influence index $\Phi_{\mathrm{B}}(f,S)$ can be obtained from the same approximation problem simply by considering the difference $f_S(S)-f_S(\varnothing)$. In Section 3 we introduce a class of weighted Banzhaf influence indexes from the solution of a weighted version of this approximation problem. We show that these indexes define a subclass of the family of \emph{generalized values}, give their most important properties, and point out similarities between the weighted Banzhaf influence index and the corresponding weighted Banzhaf interaction index. In Section 4 we discuss the issue of representing pseudo-Boolean functions in terms of Banzhaf influence indexes. More precisely, we show that in the generic weighted case any pseudo-Boolean function can be reconstructed, up to an additive constant, from prescribed influences. By contrast, in the non-weighted case only half of the information contained in the pseudo-Boolean function can be reconstructed. This important observation fully motivates the investigation of the weighted case, which therefore is not a straightforward extension of the non-weighted case. Finally, in Section 5 we present an application of the weighted Banzhaf influence index in system reliability theory and give a couple of concluding remarks.

\section{Interactions, influences, and least squares approximations}\label{sec:2}

In this section we recall how the Banzhaf interaction index can be obtained from the solution of a standard least squares approximation problem and we show how a variant of this approximation problem can be used to define both the Banzhaf interaction and influence indexes.

It is well known (see, e.g., \cite{HamRud68}) that any pseudo-Boolean function $f\colon\{0,1\}^n\to\R$ can be uniquely represented by a multilinear polynomial function
$$
f ~=~ \sum_{T\subseteq N} a(T)\, u_T\, ,
$$
where $u_T(\bfx)=\prod_{i\in T}x_i$ is the \emph{unanimity game} (or \emph{unanimity function}) for $T\subseteq N$ (with the convention $u_\varnothing=1$) and the set function $a\colon 2^{N}\to\R$, called the \emph{M\"obius transform} of $f$, is defined through the conversion formulas (M\"obius inversion formulas)
\begin{equation}\label{eq:Mob}
a(S) ~=~ \sum_{T\subseteq S} (-1)^{|S|-|T|}\, f(T)\quad\mbox{and}\quad f(S) ~=~ \sum_{T\subseteq S}\, a(T)\, .
\end{equation}

By extending formally any pseudo-Boolean function $f\colon\{0,1\}^n\to\R$ to the unit hypercube $[0,1]^n$ by linear interpolation, Owen~\cite{Owe72,Owe88} introduced the \emph{multilinear extension} of $f$, i.e., the multilinear polynomial $\bar{f}\colon [0,1]^n\to\R$ defined by
$$
\bar{f}(\bfx) ~=~ \sum_{S\subseteq N} a(S)\,\prod_{i\in S}x_i\, ,
$$
where $a$ is the M\"obius transform of $f$.

Denote by $\mathcal{F}^{N}$ the set of pseudo-Boolean functions on $N$ (i.e., with variables in $N$). Recall that the \emph{Banzhaf interaction index} \cite{GraRou99,Rou96} is the mapping $I_{\mathrm{B}}\colon\mathcal{F}^{N}\times 2^{N}\to\R$ defined in Eq.~(\ref{eq:d7df5fd}). Extending the $S$-difference operator $\Delta_S$ to multilinear polynomials on $[0,1]^n$, we can show the following identities (see~\cite{GraMarRou00,Owe88})
$$
I_\mathrm{B}(f,S) ~=~ (\Delta_S\bar{f})\Big(\boldsymbol{\frac{1}{2}}\Big) ~=~ \int_{[0,1]^n}\Delta_S\bar{f}(\bfx)\, d\bfx\, ,
$$
where $\boldsymbol{\frac 12}$ stands for $\big(\frac 12,\ldots,\frac 12\big)$. Since the $S$-difference operator has the same effect as the $S$-derivative operator $D_S$ (i.e., the partial derivative operator
with respect to the variables in $S$) when applied to multilinear polynomials on $[0,1]^n$, we also have
\begin{equation}\label{eq:we7r89a}
I_\mathrm{B}(f,S) ~=~ (D_S\bar{f})\Big(\boldsymbol{\frac 12}\Big) ~=~ \int_{[0,1]^n}D_S\bar{f}(\bfx)\, d\bfx\, .
\end{equation}

We now recall how the index $I_{\mathrm{B}}$ can be obtained from an approximation problem. For $k\in\{0,\ldots,n\}$ define
$$
V_k ~=~ \mathrm{span}\{u_T:T\subseteq N,\, |T|\leqslant k\}\, ,
$$
that is, $V_k$ is the linear subspace of all multilinear polynomials $g\colon\{0,1\}^n\to\R$ of degree at most $k$, i.e., of the form
$$
g ~=~ \sum_{\textstyle{T\subseteq N\atop |T|\leqslant k}} c(T)\,u_T\, ,\qquad c(T)\in\R\, .
$$
The best $k$th approximation of a function $f\colon\{0,1\}^n\to\R$ is the function $f_k\in V_k$ that minimizes the squared distance
\begin{equation}\label{eq:NonWeiDist}
\sum_{\bfx\in\{0,1\}^n}\big(f(\bfx)-g(\bfx)\big)^2 ~=~ \sum_{T\subseteq N}\big(f(T)-g(T)\big)^2
\end{equation}
among all functions $g\in V_k$.

The following proposition, which was proved in \cite{GraMarRou00} (see \cite{HamHol92} for an earlier work), expresses the number $I_{\mathrm{B}}(f,S)$ in terms of the best $|S|$th approximation $f_{|S|}$ of $f$.

\begin{proposition}[{\cite{GraMarRou00}}]\label{prop:interproj6}
For every $f\colon\{0,1\}^n\to\R$ and every $S\subseteq N$, the number $I_{\mathrm{B}}(f,S)$ is the coefficient of $u_S$ in the multilinear representation of the best $|S|$th approximation $f_{|S|}$ of $f$.
\end{proposition}

An alternative (and perhaps more natural) approach to measure the influence on $f$ of its $i$th variable consists in considering the coefficient of $u_{\{i\}}$ in the best
approximation of $f$ by a function of the form
$$
g=c(\varnothing)\, u_{\varnothing}+ c(\{i\})\, u_{\{i\}}
$$
(instead of a function in $V_1$), as classically done for linear models in statistics. More generally, for every $S\subseteq N$ define $V_S=\{u_T:T\subseteq S\}$, that is, $V_S$ is the linear subspace of all multilinear polynomials $g\colon\{0,1\}^n\to\R$ that depend only on the variables in $S$, i.e., of the form
$$
g ~=~ \sum_{T\subseteq S} c(T)\,u_T\, ,\qquad c(T)\in\R\, .
$$
The \emph{best $S$-approximation} of a function $f\colon\{0,1\}^n\to\R$ is then the function $f_S\in V_S$ that minimizes the squared distance (\ref{eq:NonWeiDist})
among all functions $g\in V_S$.

We now show that $I_{\mathrm{B}}(f,S)$ is also the coefficient of $u_S$ in the multilinear representation of $f_S$. On the one hand, $f_S$ is the orthogonal projection of $f$ onto $V_S$ with respect to the inner product
\begin{equation}\label{scalprod}
\langle f,g\rangle ~=~ \frac{1}{2^n}\sum_{T\subseteq N} f(T)\, g(T)\, .\,\footnote{Note that the multiplicative normalization of the inner product does not
change the projection problem.}
\end{equation}
On the other hand, it is well known and easy to prove that the $2^n$ functions
$$
v_T(\bfx)=\prod_{i\in T}(2x_i-1),\qquad T\subseteq N,
$$
form an orthonormal set with respect to this inner product. Thus, the
best $k$th- and $S$-approximations of $f$ are respectively given by
\begin{equation}\label{projAS}
f_k ~=~ \sum_{\textstyle{T\subseteq N\atop |T|\leqslant k}}\langle f, v_T\rangle\, v_T\quad\mbox{and}\quad f_S ~=~ \sum_{T\subseteq S}\langle f,
v_T\rangle\, v_T\, .
\end{equation}

These formulas enable us to prove the following simple but important result, which expresses the number $I_{\mathrm{B}}(f,S)$ in terms of the best $S$-approximation $f_S$ of $f$.

\begin{proposition}\label{prop:interproj}
For every $f\colon\{0,1\}^n\to\R$ and every $S\subseteq N$, the number $I_{\mathrm{B}}(f,S)$ is the coefficient of $u_S$ (i.e., the leading coefficient) in the multilinear representation of the best $S$-approximation $f_S$ of $f$.
\end{proposition}

\begin{proof}
Since $I_{\mathrm{B}}(f,S)$ is the coefficient of $u_S$ in the multilinear representation of $f_{|S|}$, from the first equality in (\ref{projAS}) we obtain
\begin{equation}\label{scalprodinter}
I_{\mathrm{B}}(f,S) ~=~ 2^{|S|}\, \langle f,v_S\rangle\, .
\end{equation}
We then conclude by the second equality in (\ref{projAS}).
\end{proof}

Thus, combining Proposition~\ref{prop:interproj} with Eq.~(\ref{eq:Mob}), we immediately see that the number $I_{\mathrm{B}}(f,S)$ can be expressed in terms of the approximation $f_S$ as
$$
I_{\mathrm{B}}(f,S) ~=~ \sum_{T\subseteq S}(-1)^{|S|-|T|}\, f_S(T)\, .
$$

Recall that the \emph{Banzhaf influence index} \cite{MarKojFuj07} is the mapping $\Phi_{\mathrm{B}}\colon\mathcal{F}^{N}\times 2^{N}\to\R$ defined by
\begin{equation}\label{Influence1}
\Phi_{\mathrm{B}}(f,S) ~=~ \frac{1}{2^{n-|S|}}\sum_{T\subseteq N\setminus S}\big(f(T\cup S)-f(T)\big)\, .
\end{equation}

Since the map $f\mapsto \Phi_{\mathrm{B}}(f,S)$ is linear for every $S\subseteq N$, it can be expressed by means of the inner product (\ref{scalprod}). To this aim, consider the function $g_{S}\colon\{0,1\}^n\to\R$ defined by
\begin{equation}\label{hsgs}
g_S(\bfx) ~=~ 2^{|S|}\, \bigg(\prod_{i\in S}x_i-\prod_{i\in S}(1-x_i)\bigg)\, .
\end{equation}

\begin{proposition}\label{prop:pscalIm}
For every $f\colon\{0,1\}^n\to\R$ and every $S\subseteq N$, we have $\Phi_{\mathrm{B}}(f,S)=\langle f,g_S\rangle$.
\end{proposition}

\begin{proof}
Using (\ref{scalprod}), we obtain
\[
\langle f,g_S\rangle ~=~ \frac{1}{2^{n-|S|}}\Big(\sum_{T\supseteq S} f(T)-\sum_{T\subseteq N\setminus S} f(T)\Big)\, ,
\]
which is precisely the right-hand side of (\ref{Influence1}).
\end{proof}

From Proposition~\ref{prop:pscalIm} we can easily derive an explicit expression for $\Phi_{\mathrm{B}}(f,S)$ in terms of the Banzhaf interaction index $I_{\mathrm{B}}$. This
expression was already found in \cite{Mar00}. We first consider a lemma.

\begin{lemma}\label{lemma:d7f6}
For every $S\subseteq N$, we have $g_S=2\,\sum_{T\subseteq S,\, |T|\,\mathrm{odd}}v_T$.
\end{lemma}

\begin{proof}
Since the functions $v_T$ $(T\subseteq N)$ form an orthonormal basis for $\mathcal{F}^{N}$, we have $g_S= \sum_{T\subseteq N}\langle g_S,
v_T\rangle\, v_T$. Using (\ref{scalprodinter}), (\ref{hsgs}), and then (\ref{eq:we7r89a}), we obtain
\[
\langle g_S, v_T\rangle ~=~ 2^{-|T|}\, I_{\mathrm{B}}(g_S,T) ~=~ 2^{-|T|}\, (D_T\bar g_S)\Big(\boldsymbol{\frac{1}{2}}\Big)\, .
\]
The result then follows directly from the computation of the derivative $D_T\bar g_S$.
\end{proof}

\begin{proposition}[{\cite[Proposition~4.1]{Mar00}}]\label{propinfinter}
For every $f\colon\{0,1\}^n\to\R$ and every $S\subseteq N$, we have
$$
\Phi_{\mathrm{B}}(f,S) ~=~ \sum_{\textstyle{T\subseteq S\atop |T|\,\mathrm{odd}}}\Big(\frac 12\Big)^{|T|-1}\, I_{\mathrm{B}}(f,T)\, .
$$
\end{proposition}

\begin{proof}
By Proposition~\ref{prop:pscalIm} and Lemma~\ref{lemma:d7f6}, we obtain
\begin{equation}\label{sdf6fs6sd}
\Phi_{\mathrm{B}}(f,S) ~=~ \langle f,g_S\rangle ~=~ 2\sum_{\textstyle{T\subseteq S\atop |T|\,\mathrm{odd}}}\langle f,v_T\rangle\, .
\end{equation}
We then conclude by (\ref{scalprodinter}).
\end{proof}

The following proposition gives an expression for  $\Phi_{\mathrm{B}}(f,S)$ in terms of the best $S$-approximation $f_S$ of $f$. This proposition together with Proposition~\ref{prop:interproj} show that the indexes $I_{\mathrm{B}}(f,S)$ and $\Phi_{\mathrm{B}}(f,S)$ are actually two facets of the same construction, namely the best $S$-approximation of $f$.

\begin{proposition}\label{prop:influenceapprox}
For every $f\colon\{0,1\}^n\to\R$ and every $S\subseteq N$, we have $$\Phi_{\mathrm{B}}(f,S)~=~f_S(S)-f_S(\varnothing).$$
\end{proposition}

\begin{proof}
By (\ref{projAS}), we have
\[
f_S(S)-f_S(\varnothing) ~=~ \sum_{T\subseteq S}\langle f, v_T\rangle \big(v_T(S)-v_T(\varnothing)\big) ~=~ \sum_{T\subseteq S}\langle f, v_T\rangle\big(1-(-1)^{|T|}\big)\, .
\]
Using (\ref{sdf6fs6sd}), we see that the latter expression is precisely $\Phi_{\mathrm{B}}(f,S)$.
\end{proof}

Proposition~\ref{prop:influenceapprox} is actually one of the key results of this paper. Indeed, as we will now see, it will enable us to define weighted Banzhaf influence indexes from a weighted version of the approximation problem in complete analogy with the way the weighted Banzhaf interaction index was defined in \cite{MarMat11}.

\section{Weighted influences defined by least squares}

In \cite{MarMat11} we investigated weighted versions of the best $k$th approximation problem for pseudo-Boolean functions (e.g., to allow nonuniform assignments of the variables). This study enabled us to define a class of weighted Banzhaf interaction indexes. In the present section we show that the corresponding weighted version of the best $S$-approximation problem described in Section~\ref{sec:2} not only yields the same weighted Banzhaf interaction index but also provides a natural definition of a weighted Banzhaf influence index.

Given a weight function $w\colon\{0,1\}^n\to\left]0,\infty\right[$ and a pseudo-Boolean function $f\colon\{0,1\}^n\to\R$, we define the \emph{best $S$-approximation of $f$} as the unique multilinear polynomial in $V_S$ that minimizes the squared distance
\begin{equation}\label{eq:WeiDist}
\sum_{\bfx\in\{0,1\}^n}w(\bfx)\big(f(\bfx)-g(\bfx)\big)^2 ~=~ \sum_{T\subseteq N}w(T)\big(f(T)-g(T)\big)^2
\end{equation}
among all functions $g\in V_S$.

Assuming without loss of generality that $\sum_{T\subseteq N}w(T)=1$, we see that $w$ defines a probability distribution over $2^{N}$. Considering the game theory context, we can interpret $w(T)$ as the probability that coalition $T$ forms, that is, $w(T)=\Pr(C=T)$, where $C$ represents a random coalition.

We also assume that the variables are set independently of each other. In game theory, this means that the players behave independently of each other to form coalitions, i.e., the events $(C\ni i)$ ($i\in N$) are independent.\footnote{In Section~\ref{sec:cr5} we give a justification for this independence assumption.} Setting $p_i=\Pr(C\ni i)=\sum_{S\ni i}w(S)$, we then have
\begin{equation}\label{eq:indep}
w(S) ~=~ \prod_{i\in S}p_i\,\prod_{i\in N\setminus S}(1-p_i)\, ,
\end{equation}
which implies $0<p_i<1$. Thus, the probability distribution $w$ is completely determined by the $n$-tuple $\bfp=(p_1,\ldots,p_n)\in\left]0,1\right[^n$.

We now provide an explicit expression for the best $S$-approximation of a pseudo-Boolean function. On the one hand, the squared distance (\ref{eq:WeiDist}) is induced by the weighted Euclidean inner product
$$
\langle f,g\rangle ~=~ \sum_{\bfx\in\{0,1\}^n}w(\bfx)\, f(\bfx)\,g(\bfx)\, .
$$
On the other hand, as observed in \cite{DinLaxCheChe10} the functions $v_{T,\bfp}\colon\{0,1\}^n\to\R$ ($T\subseteq N$) defined by
\begin{equation}\label{eq:asfkS2}
v_{T,\bfp}(\bfx) ~=~ \prod_{i\in T}\frac{x_i-p_i}{\sqrt{p_i(1-p_i)}}
\end{equation}
are pairwise orthogonal and normalized. This provides the following immediate solution to the weighted approximation problem.

\begin{proposition}\label{thm:Coeffk}
The best $S$-approximation of $f\colon\{0,1\}^n\to\R$ is given by
\begin{equation}\label{eq:akS2}
f_{S,\p} ~=~ \sum_{T\subseteq S}\langle f,v_{T,\bfp}\rangle \, v_{T,\bfp}\, .
\end{equation}
\end{proposition}

From Proposition~\ref{thm:Coeffk} we immediately deduce that the coefficient of $u_S$ (i.e., the leading coefficient) in the multilinear representation of $f_{S,\p}$ is given by
\begin{equation}\label{eq:sd56fsd}
I_{\mathrm{B},\bfp}(f,S) ~=~ \frac{\langle f, v_{S,\bfp}\rangle}{\prod_{i\in S}\sqrt{p_i(1-p_i)}}~,
\end{equation}
which is precisely the weighted Banzhaf interaction index introduced in \cite{MarMat11} by means of the corresponding $k$th approximation problem. In the non-weighted case (i.e., when $\bfp=\boldsymbol{\frac{1}{2}}$), Eq.~(\ref{eq:sd56fsd}) reduces to (\ref{scalprodinter}).

By analogy with Proposition \ref{prop:influenceapprox} we now propose the following definition of weighted Banzhaf influence index.

\begin{definition}\label{def:df67f5}
Let $\Phi_{\mathrm{B},\bfp}\colon\mathcal{F}^N\times 2^N\to\R$ be defined as $\Phi_{\mathrm{B},\bfp}(f,S)=f_{S,\p}(S)-f_{S,\p}(\varnothing).$
\end{definition}

We now provide various explicit expressions for $\Phi_{\mathrm{B},\bfp}(f,S)$ in terms of the weighted Banzhaf interaction index, the M\"obius transform of $f$, and the $f$ values.

We start with the following result, which is the weighted counterpart of Proposition~\ref{propinfinter}.

\begin{proposition}\label{prop:752}
For every $f\colon\{0,1\}^n\to\R$ and every $S\subseteq N$, we have
\begin{equation}\label{eq:ds6fss}
\Phi_{\mathrm{B},\bfp}(f,S) ~=~ \sum_{T\subseteq S}I_{\mathrm{B},\p}(f,T)\,\bigg(\prod_{i\in T}(1-p_i)-(-1)^{|T|}\,\prod_{i\in T}p_i\bigg)\, .
\end{equation}
\end{proposition}

\begin{proof}
Using Definition~\ref{def:df67f5} and Eqs.~(\ref{eq:akS2}) and (\ref{eq:asfkS2}), we obtain
\begin{equation}\label{eq:sd67sdfdf89}
\Phi_{\mathrm{B},\bfp}(f,S) ~=~ \sum_{T\subseteq S}\langle f,v_{T,\bfp}\rangle\,\bigg(\prod_{i\in T}\frac{1-p_i}{\sqrt{p_i(1-p_i)}}-(-1)^{|T|}\,\prod_{i\in T}\frac{p_i}{\sqrt{p_i(1-p_i)}}\bigg)\, .
\end{equation}
We then conclude by (\ref{eq:sd56fsd}).
\end{proof}

Using the expression of the weighted Banzhaf interaction index in terms of the M\"obius transform of $f$, that is,
\begin{equation}\label{eq:sdf4sdf}
I_{\mathrm{B},\p}(f,S) ~=~ \sum_{T\supseteq S}a(T)\,\prod_{i\in T\setminus S}p_i
\end{equation}
(see \cite{MarMat11}), we can obtain the corresponding expression for the weighted Banzhaf influence index. To this extent, recall the binomial product formula
\begin{equation}\label{eq:BinProdForm56}
\sum_{T\subseteq N}\prod_{i\in T}a_i\,\prod_{i\in N\setminus T}b_i ~=~ \prod_{i\in N}(a_i+b_i)\, .
\end{equation}

\begin{proposition}\label{prop:753}
For every $f\colon\{0,1\}^n\to\R$ and every $S\subseteq N$, we have
\begin{equation}\label{eq:as5d4d}
\Phi_{\mathrm{B},\bfp}(f,S) ~=~ \sum_{\textstyle{T\subseteq N\atop T\cap S\neq\varnothing}}a(T)\,\prod_{i\in T\setminus S}p_i\, .
\end{equation}
\end{proposition}

\begin{proof}
Combining (\ref{eq:ds6fss}) with (\ref{eq:sdf4sdf}), we obtain
\begin{eqnarray}
\Phi_{\mathrm{B},\bfp}(f,S) &=& \sum_{R\subseteq S}\,\sum_{T\supseteq R}a(T)\,\prod_{i\in T\setminus R}p_i\,\bigg(\prod_{i\in R}(1-p_i)-\prod_{i\in R}(-p_i)\bigg)\nonumber\\
&=& \sum_{\textstyle{T\subseteq N\atop T\cap S\neq\varnothing}}a(T)\,\prod_{i\in T\setminus S}p_i\,\sum_{R\subseteq T\cap S}\,\prod_{i\in (T\cap S)\setminus R}p_i\,\bigg(\prod_{i\in R}(1-p_i)-\prod_{i\in R}(-p_i)\bigg)\, .\label{eq:sdff43}
\end{eqnarray}
Using the binomial product formula (\ref{eq:BinProdForm56}), we see that the inner sum in (\ref{eq:sdff43}) becomes $1-\prod_{i\in T\cap S}(p_i-p_i)=1$. This completes the proof of the proposition.
\end{proof}

Interestingly, Eqs.~(\ref{eq:sdf4sdf}) and (\ref{eq:as5d4d}) show that both $I_{\mathrm{B},\p}(f,S)$ and $\Phi_{\mathrm{B},\p}(f,S)$ are independent of those $p_i$ such that $i\in S$.

A \emph{generalized value} \cite{MarKojFuj07} is a mapping $G\colon\mathcal{F}^N\times 2^{N}\to\R$ defined by
\begin{equation}\label{gv}
G(f,S) ~=~ \sum_{T\subseteq N\setminus S}p_T^S(f(T\cup S)-f(T))\, ,
\end{equation}
where the coefficients $p_T^S$ are real numbers for every $S\subseteq N$ and every $T\subseteq N\setminus S$.

The following lemma gives an expression for $G(f,S)$ in terms of the M\"obius transform of $f$. The proof is given in Appendix~\ref{sec:aa}.

\begin{lemma}\label{lem:pq}
A mapping $G\colon \mathcal{F}^N\times 2^{N}\to\R$ of the form
\begin{equation}\label{gv1}
G(f,S) ~=~ \sum_{\textstyle{R\subseteq N\atop R\cap S\neq\varnothing}} q_R^S\, a(R)\, ,
\end{equation}
where $a$ is the M\"obius transform of $f$, defines a generalized value if and only if the coefficients $q_R^S$ depend only on $S$ and $R\setminus S$. In this case, the conversion between
(\ref{gv}) and (\ref{gv1}) is given by
\[
q_R^S ~=~\sum_{T:R\setminus S\subseteq T\subseteq N\setminus S} p_T^S\quad\mbox{and}\quad p_T^S ~=~ \sum_{R: T\subseteq R\subseteq N\setminus S}(-1)^{|R|-|T|}\, q_{R\cup S}^S\, .
\]
\end{lemma}

The following proposition shows that the weighted Banzhaf influence index $\Phi_{\mathrm{B},\p}$ is a particular generalized value.

\begin{proposition}\label{prop:756}
For every $f\colon\{0,1\}^n\to\R$ and every $S\subseteq N$, we have
\[
\Phi_{\mathrm{B},\p}(f,S) ~=~ \sum_{T\subseteq N\setminus S}p_T^S\,\big(f(T\cup S)-f(T)\big)\, ,
\]
where the coefficients
\begin{equation}\label{eq:ptsn}
p_T^S ~=~ \prod_{i\in T}p_i\,\prod_{i\in N\setminus (S\cup T)}(1-p_i)
\end{equation}
satisfy the conditions $p_T^S\geqslant 0$ and $\sum_{T\subseteq N\setminus S}p_T^S=1$.
\end{proposition}

\begin{proof}
Proposition~\ref{prop:753} and Lemma~\ref{lem:pq} show that $\Phi_{\mathrm{B},\p}$ is a generalized value with $q_R^S=\prod_{i\in R\setminus S}p_i$. By Lemma~\ref{lem:pq} we then have
\[
p_T^S ~=~ \sum_{R: T\subseteq R\subseteq N\setminus S}(-1)^{|R|-|T|}\,\prod_{i\in R}p_i ~=~ \prod_{i\in T}p_i\sum_{R: T\subseteq R\subseteq N\setminus S}\,\prod_{i\in R\setminus T}(-p_i).
\]
The result then follows from the binomial product formula (\ref{eq:BinProdForm56}).
\end{proof}

The coefficients $p_T^S$ given in (\ref{eq:ptsn}) coincide with those of the corresponding expression for the weighted Banzhaf interaction index (see \cite[Theorem~10]{MarMat11}). Therefore, we immediately derive the following interpretations of these coefficients (see \cite[Proposition~11]{MarMat11}). For every $S\subseteq N$ and every $T\subseteq N\setminus S$, we have
$$
p_T^S ~=~ \Pr(T\subseteq C\subseteq S\cup T) ~=~ \Pr(C=S\cup T\mid C\supseteq S) ~=~ \Pr(C=T\mid C\subseteq N\setminus S)\, ,
$$
where $C$ denotes a random coalition.

For every $S\subseteq N$, define the linear operator $\sigma_S$ for functions on $\{0,1\}^n$ or $[0,1]^n$ by
$$
\sigma_Sf(\bfx) ~=~ f(\bfx\mid x_i=1\,\forall i\in S)-f(\bfx\mid x_i=0\,\forall i\in S)\, .
$$
For instance, when applied to the unanimity game $u_T$ ($T\subseteq N$), we obtain
\begin{equation}\label{eq:sdsfsf}
\sigma_Su_T ~=~
\begin{cases}
u_{T\setminus S}\, , & \mbox{if $S\cap T\neq\varnothing$}\, ,\\
0\, , & \mbox{otherwise}\, .
\end{cases}
\end{equation}

The next result gives various expressions for $\Phi_{\mathrm{B},\p}(f,S)$ in terms of the function $\sigma_Sf$. Recall first that, for every function $f\colon\{0,1\}^n\to\R$, we have
\begin{equation}\label{eq:5szd4f}
\bar{f}(\p) ~=~ \sum_{\bfx\in\{0,1\}^n}w(\bfx)\,f(\bfx) ~=~ E[f(C)]\, ,
\end{equation}
where $C$ denotes a random coalition  (see \cite{Owe88} or \cite[Proposition~4]{MarMat11}).

\begin{proposition}
For every $f\colon\{0,1\}^n\to\R$ and every $S\subseteq N$, we have
\begin{equation}\label{eq:ffdsf5}
\Phi_{\mathrm{B},\p}(f,S) ~=~ (\sigma_S\bar{f})(\p) ~=~ \sum_{\bfx\in\{0,1\}^n}w(\bfx)\,\sigma_Sf(\bfx) ~=~ E\big[(\sigma_Sf)(C)\big]\, ,
\end{equation}
where $C$ denotes a random coalition.
\end{proposition}

\begin{proof}
The first equality immediately follows from Eqs.~(\ref{eq:as5d4d}) and (\ref{eq:sdsfsf}). The other equalities immediately follow from (\ref{eq:5szd4f}).
\end{proof}

Interestingly, (\ref{eq:ffdsf5}) shows a strong analogy with the identities (see \cite[Propositions~4 and 9]{MarMat11})
\begin{equation}\label{eq:ffdsf51}
I_{\mathrm{B},\p}(f,S) ~=~ (D_S\bar{f})(\p) ~=~ \sum_{\bfx\in\{0,1\}^n}w(\bfx)\,\Delta_Sf(\bfx) ~=~ E\big[(\Delta_Sf)(C)\big]\, .
\end{equation}

We also have the following expression for $\Phi_{\mathrm{B},\p}(f,S)$ as an integral. We omit the proof since it follows exactly the same steps as in the proof of the corresponding expression for $I_{\mathrm{B},\p}(f,S)$ (see \cite[Proposition 12]{MarMat11}).

\begin{proposition}\label{prop:jhfh43}
Let $F_1,\ldots,F_n$ be cumulative distribution functions on $[0,1]$. Then
$$
\Phi_{\mathrm{B},\mathbf{p}}(f,S)=\int_{[0,1]^n}(\sigma_S\bar{f})(\mathbf{x})\, dF_1(x_1)\cdots dF_n(x_n)
$$
for every $f\colon\{0,1\}^n\to\R$ and every $S\subseteq N$ if and only if $p_i=\int_0^1x\, dF_i(x)$ for every $i\in N$.
\end{proposition}

We now generalize Proposition~\ref{prop:pscalIm} to the weighted case. To this aim, consider the function $g_{S,\p}\colon\{0,1\}^n\to\R$ defined by
$$
g_{S,\p}(\bfx) ~=~ \prod_{i\in S}\frac{x_i}{p_i}-\prod_{i\in S}\frac{1-x_i}{1-p_i}\, .
$$

\begin{proposition}
For every $f\colon\{0,1\}^n\to\R$ and every $S\subseteq N$, we have
\begin{equation}\label{eq:s67df}
\Phi_{\mathrm{B},\p}(f,S) ~=~ \langle f,g_{S,\p}\rangle ~=~ \sum_{\bfx\in\{0,1\}^n}w(\bfx)\, f(\bfx)\, \bigg(\prod_{i\in S}\frac{x_i}{p_i}-\prod_{i\in S}\frac{1-x_i}{1-p_i}\bigg)
\end{equation}
and
\begin{equation}\label{eq:s67df2}
\Phi_{\mathrm{B},\p}(f,S) ~=~ \sum_{\bfx\in\{0,1\}^n}f(\bfx)\,\frac{g_S(\bfx)}{2^{|S|}}\,\prod_{i\in N\setminus S}p_i^{x_i}\,(1-p_i)^{1-x_i}\, .
\end{equation}
\end{proposition}

\begin{proof}
On the one hand, by substituting (\ref{eq:asfkS2}) into (\ref{eq:sd67sdfdf89}), we obtain $\Phi_{\mathrm{B},\p}(f,S) = \langle f,g'_{S,\p}\rangle$, where
$$
g'_{S,\p}(\bfx) ~=~ \sum_{T\subseteq S}\bigg(\prod_{i\in T}\frac{x_i-p_i}{p_i}-(-1)^{|T|}\,\prod_{i\in T}\frac{x_i-p_i}{1-p_i}\bigg)\, .
$$
Using the binomial product formula (\ref{eq:BinProdForm56}), we immediately see that $g'_{S,\p}=g_{S,\p}$, which proves (\ref{eq:s67df}).

On the other hand, for every $\bfx\in\{0,1\}^n$ we have
$$
g_{S,\p}(\bfx)\, w(\bfx) ~=~ g_{S,\p}(\bfx)\,\prod_{i\in N}p_i^{x_i}\,(1-p_i)^{1-x_i} ~=~ \frac{g_S(\bfx)}{2^{|S|}}\,\prod_{i\in N\setminus S}p_i^{x_i}\,(1-p_i)^{1-x_i}\, ,
$$
which, when combined with (\ref{eq:s67df}), immediately leads to (\ref{eq:s67df2}).
\end{proof}

We end this section by giving an interpretation of the Banzhaf influence index $\Phi_{\mathrm{B}}$ as a center of mass of weighted Banzhaf influence indexes $\Phi_{\mathrm{B},\p}$.

As already mentioned, the index $\Phi_{\mathrm{B}}$ can be expressed in terms of $\Phi_{\mathrm{B},\p}$ simply by setting $\p=\boldsymbol{\frac{1}{2}}$. However, by Proposition~\ref{prop:756} we also have the following expression
\begin{equation}\label{eq:imdfa89}
\Phi_{\mathrm{B}}(f,S) ~=~ \int_{[0,1]^n}\Phi_{\mathrm{B},\mathbf{p}}(f,S) \, d\mathbf{p}\, .
\end{equation}
This formula can be interpreted in the game theory context in the same way as the corresponding formula for the interaction index (see \cite[{\S}5.1]{MarMat11}). We have assumed that the players behave independently of each other to form coalitions, each player $i$ with probability $p_i\in\left]0,1\right[$. Assuming further that this probability is not known a priori, to define an influence index it is then natural to consider the average (center of mass) of the weighted
indexes over all possible choices of the probabilities $p_i$. Eq.~(\ref{eq:imdfa89}) then shows that we obtain the non-weighted influence index $\Phi_{\mathrm{B}}$.

The \emph{Shapley generalized value} \cite{Mar00,MarKojFuj07} for a function $f\colon\{0,1\}^n\to\R$ and a coalition $S\subseteq N$ is defined by
$$
\Phi_{\mathrm{Sh}}(f,S) ~=~ \sum_{\textstyle{T\subseteq N\atop T\cap S\neq\varnothing}} \frac{a(T)}{|T\setminus S|+1}\, ,
$$
where $a$ is the M\"obius transform of $f$. Using (\ref{eq:as5d4d}) we obtain the following expression for $\Phi_{\mathrm{Sh}}$ in terms of $\Phi_{\mathrm{B},\p}$, namely
\begin{equation}\label{eq:imdfa891}
\Phi_{\mathrm{Sh}}(f,S) ~=~ \int_0^1\Phi_{\mathrm{B},(p,\ldots,p)}(f,S) \, dp\, .
\end{equation}
Here the players still behave independently of each other to form coalitions but with the same probability $p$. The integral in (\ref{eq:imdfa891}) simply represents the average of the weighted indexes over all the possible probabilities.

\section{Weighted influences as alternative representations of pseudo-Boolean functions}

It is well known that the values $I_{\mathrm{B}}(f,S)$ ($S\subseteq N$) of the non-weighted Banzhaf
interaction index for a function $f\colon\{0,1\}^n\to\R$ provide an alternative representation of $f$
(see \cite{GraMarRou00}). This observation still holds in the weighted case. Indeed, combining the Taylor expansion
formula with (\ref{eq:ffdsf51}) yields (see \cite[Eq.~(16)]{MarMat11})
\begin{equation}\label{eq:7sfd6sd}
f(\bfx) ~=~ \sum_{S\subseteq N} I_{\mathrm{B},\p}(f,S)\, \prod_{i\in S}(x_i-p_i)\, .
\end{equation}
Thus, for every $\p$ the map $f\mapsto\{I_{\mathrm{B},\p}(f,S):S\subseteq N\}$ is a linear bijection.

In this section we discuss the issue of representing pseudo-Boolean functions in terms of Banzhaf influence indexes.
In fact, we compare the non-weighted and weighted versions of the Banzhaf influence indexes and show that they have different
behaviors in terms of reconstruction of the original pseudo-Boolean function from prescribed influences. In the non-weighted version we show that the index
is degenerate: roughly speaking, the values $\Phi_{\mathrm{B}}(f,S)$ ($\varnothing\neq S\subseteq N$) encode only half of the
information contained in the function $f$. In contrast, in the weighted version, for a generic weight $\p$ the values $\Phi_{\mathrm{B},\p}(f,S)$
($\varnothing\neq S\subseteq N$) allow to reconstruct $f$ up to an additive constant.

The degeneracy of the non-weighted influence index $\Phi_{\mathrm{B}}$ follows from linear relations among
the linear functionals $\Phi_{\mathrm{B}}({\,}\cdot{\,},S)$ ($S\subseteq N$) on the space $\mathcal{F}^N$.
For instance, for every $i,j\in N$ we have $g_{\{i,j\}}=g_{\{i\}}+g_{\{j\}}$, which, by Proposition~\ref{prop:pscalIm}, translates into
$$
\Phi_{\mathrm{B}}({\,}\cdot{\,},\{i,j\})~=~\Phi_{\mathrm{B}}({\,}\cdot{\,},\{i\})+\Phi_{\mathrm{B}}({\,}\cdot{\,},\{j\})\, ,\qquad i,j\in N\, .
$$
The following result generalizes this linear dependence relation.

\begin{proposition}\label{prop:interinfl}
For every $f\colon\{0,1\}^n\to\R$ and every $S\subseteq N$, we have
\begin{equation}\label{eq:as24s}
I_{\mathrm{B},\bfp}(f,S)\,\big(v_{S,\bfp}(S)-v_{S,\bfp}(\varnothing)\big)\prod_{i\in S}\sqrt{p_i(1-p_i)} ~=~ \sum_{T\subseteq S}(-1)^{|S|-|T|}\,\Phi_{\mathrm{B},\p}(f,T)\, .
\end{equation}
\end{proposition}
\begin{proof}
Just apply the M\"obius inversion formula to (\ref{eq:ds6fss}).
\end{proof}
Formula (\ref{eq:as24s}) shows that if $\p$ is such that $(v_{S,\bfp}(S)-v_{S,\bfp}(\varnothing))=0$ for some $S\in 2^N\setminus\{\varnothing\}$, the linear functional $\Phi_{\mathrm{B},\p}({\,}\cdot{\,},S)$ on the space $\mathcal{F}^N$ is a linear combination of the functionals $\Phi_{\mathrm{B},\p}({\,}\cdot{\,},T)$ for $T\subsetneq S$. Moreover, by definition we always have $\Phi_{\mathrm{B},\p}({\,}\cdot{\,},\varnothing)=0$.

Therefore replacing a pseudo-Boolean function $f$ with the values $\Phi_{\mathrm{B},\p}(f,S)$ ($S\subseteq N$) results in a loss of information which depends on $\p$. Assuming a total order on $2^N$, we may regard $\Phi_{\mathrm{B},\p}$ as the linear map $\Phi_{\mathrm{B},\p}\colon \mathcal{F}^N\to\R^{2^n}$ defined by
$$
f\mapsto (\Phi_{\mathrm{B},\p}(f,S): S\subseteq N).
$$
We can measure the degree of dependence among the functionals $\Phi_{\mathrm{B},\p}({\,}\cdot{\,},S)$ ($S\subseteq N$) by computing the rank $\mathrm{rk}(\Phi_{\mathrm{B},\p})$ of $\Phi_{\mathrm{B},\p}$. Similarly, the resulting loss of information corresponds to the kernel $\mathrm{ker}(\Phi_{\mathrm{B},\p})$ of $\Phi_{\mathrm{B},\p}$.

\begin{proposition}\label{prop:we-99rk}
We have
$$
\mathrm{ker}(\Phi_{\mathrm{B},\p})~=~\mathrm{span}\{v_{S,\p}:S\subseteq N~\mbox{and}~v_{S,\bfp}(S)=v_{S,\bfp}(\varnothing)\}
$$
and
$$
\mathrm{rk}(\Phi_{\mathrm{B},\p}) ~=~ 2^n-|\{S\subseteq N:v_{S,\bfp}(S)=v_{S,\bfp}(\varnothing)\}|.
$$
\end{proposition}

\begin{proof}
Combining (\ref{eq:asfkS2}) with (\ref{eq:sd67sdfdf89}), we obtain
$$
\Phi_{\mathrm{B},\p}(v_{T,\p},S) ~=~
\begin{cases}
0, & \mbox{if $T\nsubseteq S$},\\
v_{T,\p}(T)-v_{T,\p}(\varnothing), & \mbox{otherwise}.
\end{cases}
$$
Thus, if $v_{T,\p}(T)-v_{T,\p}(\varnothing)=0$, then $v_{T,\p}\in\mathrm{ker}(\Phi_{\mathrm{B},\p})$. For the converse inclusion, take $f\in\mathcal{F}^N$. By (\ref{eq:7sfd6sd}), we have
$$
f ~=~ \sum_{S\subseteq N} I_{\mathrm{B},\p}(f,S)\, \prod_{i\in S}\sqrt{p_i(1-p_i)}{\,}v_{S,\p}\, .
$$
If $f\in\mathrm{ker}(\Phi_{\mathrm{B},\p})$, then $I_{\mathrm{B},\bfp}(f,S)\,(v_{S,\bfp}(S)-v_{S,\bfp}(\varnothing))=0$ for every $S\subseteq N$ by (\ref{eq:as24s}). This provides the converse inclusion. The value of $\mathrm{rk}(\Phi_{\mathrm{B},\p})$ immediately follows.
\end{proof}
We observe that the condition $v_{S,\bfp}(S)=v_{S,\bfp}(\varnothing)$ also reads
\begin{equation}\label{eq:sdf45d}
\prod_{i\in S}(1-p_i)~=~ (-1)^{|S|}\,\prod_{i\in S}p_i\, .\footnote{Or equivalently, $\prod_{i\in S}(1-1/p_i)=1$.}
\end{equation}
Since we have $\p\in\left]0,1\right[^n$, this condition cannot be fulfilled when $|S|$ is odd. Therefore by Proposition~\ref{prop:we-99rk} the rank of $\Phi_{\mathrm{B},\p}$ ranges within the interval $[2^{n-1},2^n-1]$. This motivates the following definition.

\begin{definition}
A tuple $\p\in\left]0,1\right[^n$ is \emph{nondegenerate} if
for every  $S\in 2^N\setminus\{\varnothing\}$ we have $v_{S,\bfp}(S)\neq v_{S,\bfp}(\varnothing)$, i.e., if $\mathrm{rk}(\Phi_{\mathrm{B},\p})=2^{n}-1$.
Otherwise, it is said to be \emph{degenerate}. A tuple $\p$ is \emph{maximally degenerate} if $\mathrm{rk}(\Phi_{\mathrm{B},\p})=2^{n-1}$.
\end{definition}

\begin{proposition}\label{prop:fasd65}
 The set of nondegenerate tuples is an open dense subset in $]0,1[^n$. For $n\geqslant 3$ there is a unique maximally degenerate tuple, namely $\p=\boldsymbol{\frac{1}{2}}$.
\end{proposition}

\begin{proof}
For $S\neq\varnothing$, Eq.~(\ref{eq:sdf45d}) is a nontrivial polynomial equation on the components of the tuple $\p$. This proves the first statement. To see that the second statement holds we note that $\p$ is maximally degenerate if Eq.~(\ref{eq:sdf45d}) holds for every $S$ such that $|S|$ is even. In particular it must hold for $S=\{i,j\}$, so that $p_i+p_j=1$ for all $i,j\in N$. This implies $\p=\boldsymbol{\frac{1}{2}}$ whenever $n\geqslant 3$. Finally, we can easily check that for this tuple we have $\mathrm{rk}(\Phi_{\mathrm{B},\p})=2^{n-1}$.
\end{proof}

In the following two subsections we further analyze both the maximally degenerate and nondegenerate cases.

\subsection{Behavior of the non-weighted Banzhaf influence indexes}

By Proposition~\ref{prop:fasd65} the non-weighted Banzhaf influence index $\Phi_{\mathrm{B}}$ is maximally degenerate. Let us now interpret its kernel.

\begin{definition}
Let $\ast\colon\mathcal{F}^N\to\mathcal{F}^N$ be the operator that carries $f$ into $f^*$ defined by $f^*(S)=-f(N\setminus S)$. Set also $\Sy=\{f\in\mathcal{F}^N: f^*=f\}$ and $\A=\{f\in\mathcal{F}^N: f^*=-f\}$.
\end{definition}

The spaces $\Sy$ and $\A$ can be described in terms of the functions $v_S$ as follows.

\begin{proposition}\label{prop:65f7}
We have $\mathrm{ker}(\Phi_{\mathrm{B}})=\A=\mathrm{span}\{v_S:|S|~\mbox{even}\}$ and
$\Sy=\mathrm{span}\{v_S:|S|~\mbox{odd}\}$. The space $\mathcal{F}^N$ is the direct sum of the orthogonal
subspaces $\Sy$ and $\A$. For every $S\subseteq N$, we have $g_S\in\Sy$. Finally, $\{g_S:|S|~\mbox{odd}\}$ is a basis of $\Sy$.
\end{proposition}

\begin{proof}
On the one hand, by Proposition~\ref{prop:we-99rk}, we have $\mathrm{ker}(\Phi_{\mathrm{B}})=\mathrm{span}\{v_S:|S|~\mbox{even}\}$.
On the other hand, we clearly have $v_S^*=(-1)^{|S|+1}{\,}v_S$ for every $S\subseteq N$. Therefore we have
\begin{equation}\label{eq:span}
\mathrm{span}\{v_S:|S|~\mbox{even}\}\subseteq \A\quad\mbox{and}
\quad\mathrm{span}\{v_S:|S|~\mbox{odd}\}\subseteq \Sy.
\end{equation}
It follows that $\mathrm{dim}(\A)\geqslant 2^{n-1}$ and $\mathrm{dim}(\Sy)\geqslant 2^{n-1}$. But since we have $\A\cap\Sy=\{0\}$, we must have
$\mathrm{dim}(\A)=\mathrm{dim}(\Sy)= 2^{n-1}$ and this proves the converse inclusions in (\ref{eq:span}).
This description of $\A$ and $\Sy$ proves the second assertion. The last assertions follow easily from Lemma~\ref{lemma:d7f6}.
\end{proof}

Combining Proposition~\ref{prop:pscalIm} and Eq.~(\ref{scalprodinter}) with Proposition \ref{prop:65f7} shows that the linear functionals
$\Phi_{\mathrm{B}}({\,}\cdot{\,},S)$ with $S\subseteq N$ and $I_{\mathrm{B}}({\,}\cdot{\,},S)$ with $|S|$ odd are combinations of the functionals
$\Phi_{\mathrm{B}}({\,}\cdot{\,},T)$ with $T\subseteq N$ and $|T|$ odd. These relations are given explicitly in the next proposition.

Let $E_n(x)$ denote the $n$th Euler polynomial and $E_n=2^nE_n(\frac{1}{2})$ the $n$th Euler number.

\begin{proposition}\label{prop:sd8f7}
For every $f\colon\{0,1\}^n\to\R$ and every $S\subseteq N$, we have
\begin{equation}\label{oddimportance}
\Phi_{\mathrm{B}}(f,S)=-\sum_{\textstyle{T\subseteq S\atop |T|\,\mathrm{odd}}}E_{|S|-|T|}(0)\, 2^{|S|-|T|}\, \Phi_{\mathrm{B}}(f,T)\, ,\qquad \mbox{if $|S|$ is even}\, ,
\end{equation}
and
\begin{equation}\label{evenimportance}
I_{\mathrm{B}}(f,S)=2^{|S|-1}\sum_{\textstyle{T\subseteq S\atop |T|\,\mathrm{odd}}}E_{|S|-|T|}\,\Phi_{\mathrm{B}}(f,T)\, ,\qquad \mbox{if $|S|$ is odd}\, .
\end{equation}
\end{proposition}

\begin{proof}
By Proposition~\ref{prop:pscalIm} we can prove (\ref{oddimportance}) by showing that
\begin{equation}\label{oddimportance34}
g_S ~=~ -\sum_{\textstyle{T\subseteq S\atop |T|\,\mathrm{odd}}}E_{|S|-|T|}(0)\, 2^{|S|-|T|}\, g_T\, ,\qquad \mbox{if $|S|$ is even}\, ,
\end{equation}
or equivalently (using the basic properties of Euler polynomials),
\begin{equation}\label{eq:sdfd86}
\sum_{T\subseteq S}E_{|S|-|T|}(0)\, 2^{-|T|}{\,}g_T ~=~ 0\, .
\end{equation}

To see that (\ref{eq:sdfd86}) holds, we show that
$$
\sum_{T\subseteq S}E_{|S|-|T|}(0)\, 2^{-|T|}{\,}\langle g_T,v_K\rangle ~=~ 0{\,},\qquad K\subseteq N.
$$
If $|K|$ is even, then $\langle g_T,v_K\rangle = 0$ since $g_T\in\Sy$ and $v_K\in\A$ by Proposition~\ref{prop:65f7}. If $|K|$ is odd, then by Lemma~\ref{lemma:d7f6} we have $\langle g_T,v_K\rangle = 2$ if $K\subseteq T$, and $0$, otherwise. Thus, it remains to show that
$$
\sum_{T:K\subseteq T\subseteq S}E_{|S|-|T|}(0)\, 2^{1-|T|} ~=~ 0,\qquad \mbox{for odd $|K|$}.
$$
Using the classical translation formula for Euler polynomials, we can rewrite this sum as
$$
2^{1-|K|}{\,}\sum_{t=0}^{|S|-|K|}{|S|-|K|\choose t}{\,}\Big(\frac{1}{2}\Big)^{t}{\,} E_{|S|-|K|-t}(0) ~=~ 2^{1-|K|}{\,}E_{|S|-|K|}\Big(\frac{1}{2}\Big)
$$
and the latter expression is zero since $|S|-|K|$ is odd. This completes the proof of (\ref{oddimportance}). Eq.~(\ref{evenimportance}) can be proved similarly.
\end{proof}

According to the results above, the influences $\Phi_{\mathrm{B}}(f,S)$ $(S\subseteq N)$ of a function $f\in\mathcal{F}^N$ determine only the orthogonal
projection of $f$ onto $\Sy$. On the other hand, due to Eq.~(\ref{oddimportance34}), not all vectors in $\R^{2^{n-1}}$ are influences
of a function in $\mathcal{F}^N$ : the best we can do is to build a unique function in $\Sy$ with prescribed ``odd'' influences. This is done in the following result.

\begin{proposition}\label{constructself}
For every set $\{i_T\in\R:\mbox{$|T|$ odd}\}$, the unique function $f_{\Sy}\in\Sy$ such that $\Phi_{\mathrm{B}}(f_{\Sy},T)=i_T$ for every $T\subseteq N$, $|T|$ odd, is given by
$$
f_{\Sy} ~=~ \frac{1}{2}\,\sum_{\textstyle{S\subseteq N\atop |S|\,\mathrm{odd}}}\Bigg(\sum_{\textstyle{T\subseteq S\atop |T|\,\mathrm{odd}}}E_{|S|-|T|}\, i_T\Bigg)\, v_S\, .
$$
\end{proposition}

\begin{proof}
By Proposition~\ref{prop:pscalIm}, the conditions required on $f_{\Sy}\in\Sy$ reduce to the equalities
$\langle f_{\Sy},g_T\rangle = i_T$ for odd $|T|$. Proposition~\ref{prop:65f7} then ensures existence
and uniqueness of $f_{\Sy}$. Since the set $\{v_S:|S|~\mbox{odd}\}$ is an orthonormal basis for $\Sy$ we can write
\[
f_{\Sy}~=~\sum_{\textstyle{S\subseteq N\atop |S|\,\mathrm{odd}}}\langle f_{\Sy}, v_S\rangle\, v_S\, .
\]
For odd $|S|$, by (\ref{scalprodinter}) we have $\langle f_{\Sy}, v_S\rangle=2^{-|S|}\,I_{\mathrm{B}}(f_{\Sy},S)$ and then we compute $I_{\mathrm{B}}(f_{\Sy},S)$ by using (\ref{evenimportance}).
\end{proof}

\subsection{Behavior of the weighted Banzhaf influence indexes}

The properties of the weighted influence index $\Phi_{\mathrm{B},\mathbf{p}}$ for a nondegenerate $\p$ are completely different from those of the non-weighted influence index $\Phi_{\mathrm{B}}$. By Proposition~\ref{prop:we-99rk}, for a nondegenerate $\p$ the kernel of $\Phi_{\mathrm{B},\p}$ is one-dimensional and
reduced to the constant functions. Moreover, the functionals $\Phi_{\mathrm{B},\p}({\,}\cdot{\,},S)$ for
$S\neq\varnothing$ are linearly independent. Therefore, we can build a function $f$ from its influences
$\Phi_{\mathrm{B},\p}(f,S)$ for $S\neq\varnothing$, up to an additive constant. Requiring a prescribed value of $f$ on the empty set,
or a prescribed interaction $I_{\mathrm{B},\p}(f,\varnothing)$, allows us to build a unique function. This is the aim of the next result.

\begin{proposition}\label{prop:ew7q}
Assume that $\p\in\left]0,1\right[^n$ is nondegenerate and consider a set function $i\colon 2^{N}\to\R$. There exists a unique function
$f\in\mathcal{F}^N$ such that $\Phi_{\mathrm{B},\bfp}(f,S)=i(S)$ for every nonempty $S\subseteq N$
and $I_{\mathrm{B},\p}(f,\varnothing)=i(\varnothing)$. It is given by
\[
f ~=~ i(\varnothing)+\sum_{S\neq\varnothing}\frac{v_{S,\bfp}}{v_{S,\bfp}(S)-v_{S,\bfp}(\varnothing)}\,\sum_{T\subseteq S}(-1)^{|S|-|T|}\, i(T)\, .
\]
There exists a unique set function $g\in\mathcal{F}^N$ such that $\Phi_{\mathrm{B},\bfp}(g,S)=i(S)$ for every nonempty $S\subseteq N$
and $g(\varnothing)=i(\varnothing)$. It is given by
\[
g ~=~ i(\varnothing)+\sum_{S\neq\varnothing}\frac{v_{S,\bfp}-v_{S,\bfp}(\varnothing)}{v_{S,\bfp}(S)-v_{S,\bfp}(\varnothing)}\,\sum_{T\subseteq S}(-1)^{|S|-|T|}\, i(T)\, .
\]
\end{proposition}
\begin{proof}
We compute $f$ by substituting (\ref{eq:as24s}) in (\ref{eq:7sfd6sd}). Then we have immediately $g=f-f(\varnothing)+i(\varnothing)$.
\end{proof}

\section{Application and final remarks}\label{sec:cr5}

We now end our investigation with an application of the concept of weighted Banzhaf influence index in reliability engineering. We also give a justification for our independence assumption, introduce a normalized influence index, and derive tight upper bounds on influences.

\subsection{An application in system reliability theory}

Consider a system made up of $n$ interconnected components. Let $C=\{1,\ldots,n\}$ be the set of components and let $\phi\colon\{0,1\}^n\to\{0,1\}$ be the \emph{structure function} which expresses the state of the system in terms of the states of its components. We assume that the system is \emph{semicoherent}, i.e., the structure function $\phi$ is nondecreasing in each variable and satisfies the conditions $\phi(0,\ldots,0)=0$ and $\phi(1,\ldots,1)=1$. We also assume that, at any time, the component states $X_1,\ldots,X_n$ are statistically independent. The \emph{reliability} of every  component $i\in C$ is then defined as the probability $p_i=\Pr(X_i=1)$. For general background on system reliability theory, see, e.g., Barlow and Proschan~\cite{BarPro81}.

According to the definition given by Ben-Or and Linial \cite{BenLin90} (as recalled in the introduction), for every subset $S$ of components, the index
$$
I_{\phi}(S) ~=~ \Phi_{\mathrm{B}}(\phi,S) ~=~ \frac{1}{2^{n-|S|}}{\,}\sum_{T\subseteq C\setminus S}\big(\phi(T\cup S)-\phi(T)\big)
$$
measures, at a given time, the probability that the state of the system is undetermined once the state of each component $i$ not in $S$ is set to one or zero with probability $p_i=1/2$.

In practice, however, the probabilities $\Pr(X_i=1)$ and $\Pr(X_i=0)$ need not be equal. The weighted version $\Phi_{\mathrm{B},\p}$ of the Banzhaf influence index then provides a straightforward generalization of Ben-Or and Linial's definition to the general case of arbitrary reliabilities $p_1,\ldots,p_n$. More specifically, the weighted index
\[
\Phi_{\mathrm{B},\p}(\phi,S) ~=~ \sum_{T\subseteq C\setminus S}p_T^S\,\big(\phi(T\cup S)-\phi(T)\big)\, ,
\]
where
$$
p_T^S ~=~ \prod_{i\in T}p_i\,\prod_{i\in C\setminus (S\cup T)}(1-p_i),
$$
(as described in Proposition~\ref{prop:756}) precisely measures, at a given time, the probability that the state of the system remains undetermined once the state of each component $i$ not in $S$ is set to one with probability $p_i$ and to zero with probability $1-p_i$. In a sense this probability measures, at a given time, the influence of the subset of components in $S$ over the system. When $S$ reduces to a singleton $\{i\}$ and $p_i=1/2$, we retrieve the classical Banzhaf power index, also known in reliability theory as the Birnbaum structural measure of component importance.


\subsection{On the independence assumption}

We have made the important assumption that the variables are set independently of each other. From this assumption we derived condition (\ref{eq:indep}). Let us now show that this assumption is rather natural.

For every probability distribution $w$ such that $p_i=\sum_{S\ni i}w(S)\in\left]0,1\right[$, the best $\{i\}$-approximation of $f\colon\{0,1\}^n\to\R$ with respect to the squared distance (\ref{eq:WeiDist}) associated with $w$ is given by
\[
f_{\{i\}} ~=~ \langle f, v_{\{i\},\p}\rangle\, v_{\{i\},\p}+ \langle f, 1\rangle\, ,
\]
where $v_{\{i\},\p}(\bfx)=(x_i-p_i)/\sqrt{p_i(1-p_i)}$.\footnote{Indeed, the functions $1$ and $v_{\{i\},\p}$ form an orthonormal basis for $V_{\{i\}}$.} Therefore, we can define the power/influence index associated with $w$ by
$$
I_w(f,\{i\}) ~=~ \frac{\langle f, v_{\{i\},\p}\rangle}{\sqrt{p_i(1-p_i)}} ~=~ \sum_{T\subseteq N\setminus\{i\}}\bigg(\frac{w(T\cup\{i\})}{p_i}\, f(T\cup\{i\})-\frac{w(T)}{1-p_i}\, f(T)\bigg)\, .
$$
However, we know from the literature on cooperative game theory (see, e.g., \cite{DubNeyWeb81,Web88}) that ``good'' power indexes should be of the form
\begin{equation}\label{eq:goodform}
I(f,\{i\}) ~=~ \sum_{T\subseteq N\setminus \{i\}}c_T^i\, \Delta_{\{i\}}f(T)\, ,\qquad c_T^i\in\R\, .
\end{equation}
It follows that the index $I_w({\,\cdot\,},\{i\})$ is of the form (\ref{eq:goodform}) if and only if $\frac{w(T\cup\{i\})}{p_i}=\frac{w(T)}{1-p_i}$ for every $T\subseteq N\setminus\{i\}$. Thus, we have proved the following result.

\begin{proposition}
The index $I_w({\,\cdot\,},\{i\})$ is of the form (\ref{eq:goodform}) for every $i\in N$ if and only if (\ref{eq:indep}) holds.
\end{proposition}

\subsection{Normalized index and upper bounds on influences}

Since the index $\Phi_{\mathrm{B},\mathbf{p}}$ is a linear map, it cannot be considered as an absolute
influence index but rather as a relative index constructed to assess and compare influences for a \emph{given} function.

If we want to compare influences for \emph{different} functions, we need to consider an absolute, normalized influence index. Such an index
can be defined as follows. Considering again $2^N$ as a probability space with respect to the measure $w$, we see that, for every
$S\subseteq N$ the number $\Phi_{\mathrm{B},\mathbf{p}}(f,S)$ is the covariance $\mathrm{cov}(f,g_{S,\p})$ of the random variables $f$ and $g_{S,\p}$. In fact, denoting the expectation of $f$ by $E[f]=\bar{f}(\p)$ (see (\ref{eq:5szd4f})), we have
$$
\Phi_{\mathrm{B},\mathbf{p}}(f,S) ~=~ \langle f,g_{S,\p}\rangle ~=~ \langle f-E[f],g_{S,\p}-E[g_{S,\p}]\rangle ~=~ \mathrm{cov}(f,g_{S,\p})
$$
since $E[g_{S,\p}]=\bar{g}_{S,\p}(\p)=0$ and $\langle E[f],g_{S,\p}\rangle=\Phi_{\mathrm{B},\mathbf{p}}(E[f],S)=0$.

To define a normalized influence index, we naturally consider the Pearson correlation coefficient instead of the covariance.\footnote{This approach was also considered for the interaction index (see \cite[{\S}5]{MarMat11}).} First observe that, for every nonempty subset $S\subseteq N$, the standard deviation of $g_{S,\p}$ is given by
\begin{equation}\label{eq:86ffsd}
\sigma(g_{S,\p}) ~=~ \sqrt{\prod_{i\in S}\frac{1}{p_i}+\prod_{i\in S}\frac{1}{1-p_i}}~.
\end{equation}
In fact, since $g_{S,\p}\in V_S$, we have
$$
\sigma^2(g_{S,\p}) ~=~ \mathrm{cov}(g_{S,\p},g_{S,\p}) ~=~ \Phi_{\mathrm{B},\mathbf{p}}(g_{S,\p},S) ~=~  g_{S,\p}(S)-g_{S,\p}(\varnothing)\, ,
$$
which immediately leads to (\ref{eq:86ffsd}).

\begin{definition}
The \emph{normalized influence index} is the mapping
$$
r\colon \{f\colon\{0,1\}^n\to\R :\sigma(f)\neq 0\}\times (2^N\setminus\{\varnothing\})\to\R
$$
defined by
$$
r(f,S) ~=~ \frac{\mathrm{cov}(f,g_{S,\p})}{\sigma(f)\,\sigma(g_{S,\p})} ~=~ \frac{\Phi_{\mathrm{B},\mathbf{p}}(f,S)}{\sigma(f)\,\sigma(g_{S,\p})}\, .
$$
\end{definition}

By definition the normalized influence index remains unchanged under interval scale transformations, that is, $r(af+b,S)=r(f,S)$ for all $a>0$ and $b\in\R$. Thus, it does not depend on the ``size'' of $f$ and therefore can be used to compare different functions in terms of influence.

Moreover, as a correlation coefficient, the normalized influence index satisfies the inequality $|r(f,S)|\leqslant 1$, that is,
$$
\frac{|\Phi_{\mathrm{B},\mathbf{p}}(f,S)|}{\sigma(f)} ~\leqslant ~\sigma(g_{S,\p})~.
$$
The equality holds if and only if there exist $a,b\in\R$ such that $f=a\,g_{S,\p}+b$.

Interestingly, this property shows that (\ref{eq:86ffsd}) is a tight upper bound on the influence of a normalized function $f/\sigma(f)$. Thus, for every nonempty subset $S\subseteq N$, those normalized functions for which $S$ has the greatest influence are of the form $f=(\pm{\,}g_{S,\p}+c)/\sigma(g_{S,\p})$, where $c\in\R$.

\section*{Acknowledgments}

The authors gratefully acknowledge partial support by the research project F1R-MTH-PUL-12RDO2 of the University of Luxembourg.

\appendix
\section{Proof of Lemma~\ref{lem:pq}}\label{sec:aa}

\begin{proof}[Proof of Lemma~\ref{lem:pq}]
Using the definition of the M\"obius transform in (\ref{gv}), we obtain
\begin{eqnarray*}
G(f,S) &=& \sum_{T\subseteq N\setminus S}p_T^S\,\bigg(\sum_{R\subseteq T\cup S}a(R)-\sum_{R\subseteq T}a(R)\bigg)\\
&=& \sum_{R\subseteq N}a(R)\,\bigg(\sum_{T:R\setminus S\subseteq T\subseteq N\setminus S} p_T^S-\sum_{T:R\subseteq T\subseteq N\setminus S} p_T^S\bigg)\, ,
\end{eqnarray*}
which shows that $G$ has the form (\ref{gv1}) with the prescribed $q_R^S$.

Conversely, substituting (\ref{eq:Mob}) into (\ref{gv1}) and assuming $S\neq\varnothing$, we obtain
$$
\sum_{\textstyle{R\subseteq N\atop R\cap S\neq\varnothing}} q_R^S\, a(R) ~=~ \sum_{\textstyle{R\subseteq N\atop R\cap S\neq\varnothing}} q_R^S\, \sum_{T\subseteq R}(-1)^{|R|-|T|}\, f(T) ~=~ \sum_{T\subseteq N}f(T)\,\sum_{\textstyle{R\supseteq T\atop R\cap S\neq\varnothing}}(-1)^{|R|-|T|}\, q_R^S\, .
$$
Partitioning every $R$ into $R'=R\setminus S$ and $R''=R\cap S$, the latter expression becomes
\[
\sum_{T\subseteq N}f(T)\,\sum_{T\setminus S\subseteq R'\subseteq N\setminus S}\,\sum_{\textstyle{T\cap S\subseteq R''\subseteq S\atop R''\neq\varnothing}}(-1)^{|R'|+|R''|-|T|}\, q_{R'\cup R''}^S\, .
\]
Since our assumption on the coefficients $q_R^S$ implies $q_{R'\cup R''}^S=q_{R'\cup S}^S$, the latter expression becomes
\[
\sum_{T\subseteq N}f(T)\,\sum_{T\setminus S\subseteq R'\subseteq N\setminus S}(-1)^{|R'|-|T\setminus S|}\, q_{R'\cup S}^S\,\sum_{\textstyle{T\cap S\subseteq R''\subseteq S\atop R''\neq\varnothing}}(-1)^{|R''|-|T\cap S|}\, ,
\]
where the inner sum equals $(1-1)^{|S\setminus T|}$, if $T\cap S\neq\varnothing$, and $-1$, otherwise. Setting $T'=T\setminus S$ for every $T$ containing $S$, the latter expression finally becomes
\[
\sum_{T'\subseteq N\setminus S}\bigg(\sum_{T'\subseteq R'\subseteq N\setminus S}(-1)^{|R'|-|T'|}\, q_{R'\cup S}^S\bigg)\,\big(f(T'\cup S)-f(T')\big)\, ,
\]
which completes the proof of the lemma.
\end{proof}


\end{document}